\documentclass[portuges,12pt,letter]{article}
\usepackage[centertags]{amsmath}
\usepackage{amsfonts}
\usepackage{newlfont}
\usepackage{amscd}
\usepackage{graphics}
\usepackage{epsfig}
\usepackage{indentfirst}
\usepackage{amsxtra}
\usepackage[latin1]{inputenc}
\usepackage{amssymb, amsmath}
\usepackage{amsthm}
\usepackage[mathscr]{eucal}

\newtheorem{thm}{Theorem}[section]

\title{\bf  A duality principle and related computational method for a class of structural optimization problems in elasticity}

\author{Fabio Silva Botelho$^1$ and Alexandre Molter$^2$\\ \small{$^1$Department of Mathematics, Federal University of Santa Catarina (UFSC), Florian\'{o}polis, SC - Brazil}\\ \small{$^2$Department of Mathematics and Statistics, Federal University of Pelotas (UFPel), Pelotas, RS - Brazil}}

\date{}

\begin{document}
\maketitle

\hspace{0.5cm}$\underline{\hspace{15.8cm}}$
\abstract{ In this article we develop a duality principle and concerning computational method for a structural optimization problem in elasticity.
We consider the problem of finding the optimal topology for an elastic solid which minimizes its structural inner energy resulting from the action of
external loads to be specified. The main results are obtained through standard tools of convex analysis and duality theory.
We emphasize our algorithm do not include a filter to process the results, so that the result obtained is indeed a critical point for the original optimization problem.  Finally, we present some numerical examples concerning  applications of the theoretical results established.}

\noindent\hrulefill

\section{Introduction} Consider an elastic solid which the volume corresponds to an open, bounded, connected set, denoted by $\Omega \subset \mathbb{R}^3$
with a regular (Lipschitzian) boundary denoted by $\partial \Omega=\Gamma_0 \cup \Gamma_t$ where $\Gamma_0 \cap \Gamma_t = \emptyset.$
Consider also the problem of minimizing the functional $\hat{J}:U \times B \rightarrow \mathbb{R}$ where
$$\hat{J}(u,t)=\frac{1}{2}\langle u_i,f_i\rangle_{L^2(\Omega)}+\frac{1}{2}\langle u_i,\hat{f}_i\rangle_{L^2(\Gamma_t)},$$
subject to
\begin{equation} \left\{
\begin{array}{lll}
 (H_{ijkl}(t) e_{kl}(u))_{,j}+f_i=0 \text{ in } \Omega,
 \\\\
 H_{ijkl}(t) e_{kl}(u) n_j-\hat{f}_i=0, \text{ on } \Gamma_t, \; \forall i \in \{1,2,3\}.\end{array} \right.\end{equation}
Here $\mathbf{n}=(n_1,n_2,n_3)$ denotes the outward normal to $\partial \Omega$ and  $$U=\{u=(u_1,u_2,u_3) \in W^{1,2}(\Omega;\mathbb{R}^3)\;:\; u=(0,0,0)=\mathbf{0} \text{ on } \Gamma_0\},$$

$$B=\left\{ t:\Omega \rightarrow [0,1] \text{ measurable }\;:\;\int_\Omega t(x)\;dx=t_1 |\Omega|\right\},$$
where $$0<t_1<1$$ and $|\Omega|$ denotes the Lebesgue measure of $\Omega.$

Moreover $u=(u_1,u_2,u_3) \in W^{1,2}(\Omega; \mathbb{R}^3)$ is the field of displacements relating the cartesian system $(0,x_1,x_2,x_3)$,
resulting from the action of the external loads $f \in L^2(\Omega;\mathbb{R}^3)$ and $\hat{f} \in L^2(\Gamma_t;\mathbb{R}^3).$

We also define the stress tensor $\{\sigma_{ij}\} \in Y^*=Y=L^2(\Omega; \mathbb{R}^{3 \times 3}),$ by
$$\sigma_{ij}(u)=H_{ijkl}(t) e_{kl}(u),\;$$
and the strain tensor $e:U \rightarrow L^2(\Omega;\mathbb{R}^{3 \times 3})$ by $$e_{ij}(u)=\frac{1}{2}(u_{i,j}+u_{j,i}),\; \forall i,j \in \{1,2,3\}.$$

Finally, $$\{H_{ijkl}(t)\}=\{t H_{ijkl}^0+(1-t)H_{ijkl}^1\},$$ where
$H^0$ corresponds to a strong material and $H^1$ to a very soft material, intending to simulate voids along the solid structure.

The variable $t$ is the design one, which the optimal distribution values along the structure are intended to minimize its inner work
with a volume restriction indicated through the set $B$.

 The duality principle obtained is developed inspired by the works in \cite{85, 2900}. Similar theoretical results have been developed in
\cite{12a}, however we believe the proof here presented, which is based on the min-max theorem  is easier to follow (indeed we thank an anonymous referee for his suggestion about applying the min-max theorem to complete the proof). A theory for a topology optimization problem in elasticity is presented in \cite{Xia}, even though in our book \cite{12a} of 2014, we have developed a more general
 result with a proof based on the inverse function theorem. Also, dual methods for discrete structural optimization problems were used in \cite{beck} .

 We highlight throughout this text we have used the standard Einstein sum convention of  repeated indices. Related models, among others,  are addressed in \cite{50}.

A  Matlab code using a filter for the numerical computation of similar problems is presented in \cite{3.1}. We emphasize in our algorithm we have not used a filter. The majority of topology optimization works use filtering to avoid the check-board effect \cite{51, ZuoSMO2017}. One of the proposals of this work is to increase discretization in the direction of the loads to avoid this problem.

Moreover, details on the Sobolev spaces addressed may be found in \cite{1}. In addition, the primal variational development of the topology optimization problem has been described in \cite{52}.

One of the main contributions of this work is to present  detailed theoretical developments for such a class of structural optimization problems,
through duality theory and an application of the min-max theorem. We have also discovered that without the use of any filters, to avoid the
up-surging of the check-board problem in some parts of the optimal structure, it is necessary to discretize more in the load direction, in which
the displacements are much larger.

Finally, it is  worth mentioning the numerical examples presented have been developed in a Finite Element (FE) context, based on the work of \cite{3.1}.

\section{Mathematical formulation of the topology optimization problem} 
Our mathematical topology optimization problem is summarized by the following theorem.

\begin{thm}
Consider the statements and assumptions indicated in the last section, in particular those refereing to $\Omega$ and
the functional $\hat{J} :U \times B \rightarrow \mathbb{R}.$

Define $J_1: U \times B \rightarrow \mathbb{R}$ by
$$J_1(u,t)=-G(e(u),t)+\langle u_i,f_i\rangle_{L^2(\Omega)}+\langle u_i,\hat{f}_i\rangle_{L^2(\Gamma_t)},$$
where
$$G(e(u),t)=\frac{1}{2}\int_\Omega H_{ijkl}(t) e_{ij}(u)e_{kl}(u)\;dx,$$
and where $$dx=dx_1dx_2dx_3.$$

Define also $J^*:U \rightarrow \mathbb{R}$ by
\begin{eqnarray}J^*(u)&=&\inf_{t \in B}\{J_1(u,t)\}
\nonumber \\ &=& \inf_{t \in B}\{- G(e(u),t)+ \langle u_i,f_i\rangle_{L^2(\Omega)}+\langle u_i,\hat{f}_i\rangle_{L^2(\Gamma_t)}\}.\end{eqnarray}

Assume there exists $c_0,c_1>0$ such that
$$H_{ijkl}^0 z_{ij}z_{kl} > c_0 z_{ij}z_{ij}$$ and
$$H_{ijkl}^1 z_{ij}z_{kl}>c_1 z_{ij}z_{ij}, \; \forall z=\{z_{ij}\} \in \mathbb{R}^{3 \times 3},\; \text{ such that } z \neq \mathbf{0}.$$

Finally, define $J: U \times B \rightarrow \mathbb{R} \cup\{+\infty\}$ by
$$J(u,t)=\hat{J}(u,t)+Ind(u,t),$$
where
\begin{equation} Ind(u,t)=\left\{
\begin{array}{ll}
 0,& \text{ if } (u,t) \in A^*,
 \\
 +\infty,& \text{ otherwise },\end{array} \right.\end{equation}
where $A^*=A_1 \cap A_2,$
$$A_1=\{(u,t) \in U \times B\;:\; (\sigma_{ij}(u))_{,j} +f_i=0, \text{ in } \Omega, \; \forall i \in \{1,2,3\}\}$$
and
$$A_2=\{(u,t) \in U \times B\;:\; \sigma_{ij}(u)n_j -\hat{f}_i=0, \text{ on } \Gamma_t, \; \forall i \in \{1,2,3\}\}.$$
Under such hypotheses, there exists $(u_0,t_0) \in U \times B$ such that

\begin{eqnarray}
J(u_0,t_0)&=& \inf_{(u,t) \in U \times B} J(u,t) \nonumber \\ &=& \sup_{ \hat{u} \in U} J^*(\hat{u}) \nonumber \\ &=&
J^*(u_0) \nonumber \\ &=& \hat{J}(u_0,t_0) \nonumber \\ &=& \inf_{(t,\sigma) \in B \times C^*} G^*(\sigma,t) \nonumber \\ &=& G^*(\sigma(u_0),t_0),\end{eqnarray}
where
\begin{eqnarray}G^*(\sigma,t)&=& \sup_{v \in Y} \{ \langle v_{ij},\sigma_{ij} \rangle_{L^2(\Omega)}-G(v,t)\}
\nonumber \\ &=& \frac{1}{2}\int_\Omega \overline{H}_{ijkl}(t) \sigma_{ij}\sigma_{kl}\;dx,
\end{eqnarray}
$$\{\overline{H}_{ijkl}(t)\}=\{H_{ijkl}(t)\}^{-1}$$ and
 $C^*=C_1 \cap C_2,$
 where
 $$C_1=\{ \sigma \in Y^*\;:\; \sigma_{ij,j} +f_i=0, \text{ in } \Omega, \; \forall i \in \{1,2,3\}\}$$
and
$$C_2=\{\sigma \in Y^*\;:\; \sigma_{ij}n_j -\hat{f}_i=0, \text{ on } \Gamma_t, \; \forall i \in \{1,2,3\}\}.$$
\end{thm}
\begin{proof}
Observe that
\begin{eqnarray}
\inf_{(u,t) \in U \times B} J(u,t) &=&
\inf_{t \in B} \left\{ \inf_{ u \in U} J(u,t) \right\}
\nonumber \\ &=& \inf_{t \in B}\left\{\sup_{ \hat{u} \in U} \left\{ \inf_{u \in U}\left \{ \frac{1}{2}\int_\Omega H_{ijkl}(t) e_{ij}(u)e_{kl}(u)\;dx
\right. \right.\right.
\nonumber \\ &&+\langle \hat{u}_i, (H_{ijkl}(t)e_{kl}(u))_{,j}+f_i\rangle_{L^2(\Omega)}  \nonumber \\ &&
\left.\left.\left.- \langle \hat{u}_i,H_{ijkl}(t)e_{kl}(u)n_j-\hat{f}_i \rangle_{L^2(\Gamma_t)}\right\}\right\}\right\}\nonumber \\ &=&
\inf_{t \in B}\left\{\sup_{ \hat{u} \in U} \left\{ \inf_{u \in U}\left \{ \frac{1}{2}\int_\Omega H_{ijkl}(t) e_{ij}(u)e_{kl}(u)\;dx \right.\right.\right.
\nonumber \\ &&-\int_\Omega  H_{ijkl}(t) e_{ij}(\hat{u})e_{kl}(u)\;dx \nonumber \\ && \left.\left.\left.+
\langle \hat{u}_i,f_i \rangle_{L^2(\Omega)} +\langle \hat{u}_i, \hat{f}_i\rangle_{L^2(\Gamma_t)}   \right\}\right\}\right\}\nonumber \\ &=&
 \inf_{t \in B} \left\{ \sup_{\hat{u} \in U}\left\{  -\int_\Omega H_{ijkl}(t) e_{ij}(\hat{u})e_{kl}(\hat{u})\;dx \right.\right.\nonumber \\ &&
 \left.\left.\langle \hat{u}_i,f_i \rangle_{L^2(\Omega)} +\langle \hat{u}_i, \hat{f}_i\rangle_{L^2(\Gamma_t)}\right\}\right\} \nonumber \\ &=&
 \inf_{t \in B} \left\{ \inf_{ \sigma \in C^*} G^*(\sigma,t)\right\}.
 \end{eqnarray}

 Also, from this and the min-max theorem, there exist $(u_0,t_0) \in U \times B$ such that
 \begin{eqnarray}
 \inf_{(u,t) \in U \times B} J(u,t)  &=&
 \inf_{t \in B} \left\{ \sup_{\hat{u} \in U} J_1(u,t)\right\} \nonumber \\ &=& \sup_{ u \in U} \left\{\inf_{t \in B} J_1(u,t) \right\}
 \nonumber \\ &=& J_1(u_0,t_0) \nonumber \\ &=& \inf_{t \in B}  J_1(u_0,t) \nonumber \\ &=& J^*(u_0).
 \end{eqnarray}

 Finally, from the extremal necessary condition
 $$\frac{\partial J_1(u_0,t_0)}{\partial u}=\mathbf{0}$$ we obtain
 $$ (H_{ijkl}(t_0)e_{kl}(u_0))_{,j} +f_i=0 \text{ in } \Omega,$$
 and
 $$H_{ijkl}(t_0)e_{kl}(u_0)n_j-\hat{f}_i=0 \text{ on } \Gamma_t,\; \forall i \in \{1,2,3\},$$ so that
 $$G(e(u_0))=\frac{1}{2}\langle (u_0)_i,f_i \rangle_{L^2(\Omega)} +\frac{1}{2}\langle (u_0)_i, \hat{f}_i \rangle_{L^2(\Gamma_t)}.$$

 Hence $(u_0,t_0) \in A^*$ so that $Ind(u_0,t_0)=0$ and $\sigma(u_0) \in C^*.$

 Moreover
 \begin{eqnarray}
 J^*(u_0)  &=&
 -G(e(u_0))+\langle (u_0)_i,f_i \rangle_{L^2(\Omega)} +\langle (u_0)_i, \hat{f}_i \rangle_{L^2(\Gamma_t)} \nonumber \\ &=&
 G(e(u_0)) \nonumber \\ &=& G(e(u_0))+Ind(u_0,t_0)\nonumber \\ &=& J(u_0,t_0) \nonumber \\ &=& G^*(\sigma(u_0),t_0).
 \end{eqnarray}

 This completes the proof.
 \end{proof}


 \section{About the computational method}

The continuous topology optimization problem described in the previous section is discretized using
the FE method, considering in plane deformations.  The FE discretization is performed taking into account the
bilinear isoparametric element as a master one, in similar way as in \cite{bendsoe2003, 3.1}.

To obtain computational results, we have defined the following algorithm.

 \begin{enumerate}
 \item Set $n=1$.
 \item Set $t_1(x)=t_1, \text{ in } \Omega.$
 \item\label{3} Calculate $u_n \in U$ as the solution of equation
 $$\frac{\partial J_1(u,t_n)}{\partial u}=\mathbf{0},$$ that is
 \begin{equation}\label{eq1} \left\{
\begin{array}{lll}
 (H_{ijkl}(t_n) e_{kl}(u_n))_{,j}+f_i=0 \text{ in } \Omega,
 \\\\
 H_{ijkl}(t_n) e_{kl}(u_n) n_j-\hat{f}_i=0, \text{ on } \Gamma_t, \; \forall i \in \{1,2,3\}.\end{array} \right.\end{equation}
\item\label{S1}  Obtain $t_{n+1}$ by
$$t_{n+1}=\text{arg}\min_{t \in B} J_1(u_n,t).$$
\item Set $n := n+1$ and go to  step \ref{3} up to the satisfaction of an appropriate convergence criterion.
\end{enumerate}

In the FE formulation, equations indicated in \ref{eq1} stands for 

\begin{equation}\label{matricialglobaleq}
  \mathbf{H}(t)\mathbf{U}=\mathbf{f},
\end{equation}
where $\mathbf{H}(t)$ is the global stiffness matrix, $\mathbf{U}$ is the global displacements vector and $\mathbf{f}$ is the global forces one.

 Thus, for such a FE  models ($N$ elements where $e \in \{1,...,N\}$), the primal optimization problem can be written in a matrix form as
\begin{eqnarray}\label{eq21}
\begin{array}{cl}
\text{min }  & \displaystyle \hat{J}(u,t)= \displaystyle
         \frac{1}{2}\mathbf{U}^T\mathbf{H}(t)\mathbf{U}\\
      & = \displaystyle{\frac{1}{2}\sum_{e=1}^{N}({t_{{e}}})^p\mathbf{u}_e^T\mathbf{H}_e\mathbf{u}_e} \\
\text{ subject to } & \displaystyle ({t_{{e}}})^p\mathbf{H}_e\mathbf{u}_e=\mathbf{f}_e\\
      & \displaystyle \sum_{e=1}^{N}{t_{{e}}}V_e=t_1 |\Omega| \\
      & \displaystyle 0\leq  t \leq 1 \\
      & \displaystyle e=1,2,3,...,N,
\end{array}\end{eqnarray}
On the other hand, the dual problem may be expressed by
\begin{eqnarray}\label{eq21}
\begin{array}{cl}
\text{ max } & \displaystyle J^*(u), \mbox{ where } \\
      & J^*(u)= \displaystyle{\min_{t \in B}} \left(-\frac{1}{2}\sum_{e=1}^{N}\left(({t_{{e}}})^p\mathbf{u}_e^T\mathbf{H}_e\mathbf{u}_e+\mathbf{f}_e\mathbf{u}_e\right)\right)\\
& \mbox{ where } t \in B \mbox{ if and only if }\\
      & \displaystyle \sum_{e=1}^{N}{t_{{e}}}V_e = t_1 |\Omega| \\
      & \displaystyle 0\leq t_{e} \leq  1,\\
      & \displaystyle e=1,2,3,...,N,
\end{array}\end{eqnarray}
and where $V_e$ is the area of element $e$.

Finally, the last minimization indicated corresponds to item \ref{S1} in the concerning algorithm.
Indeed, such a procedure refers to minimize at each sub-iteration, through the Matlab Linprog routine (that is, in a sequentially  linearized context), the function $$\sum_{e=1}^N\frac{\partial J_1(u_n,\{t_e^n\})}{\partial t_e}\,t_e=\sum_{e=1}^N\left(-p (t_e^n)^{p-1} t_e \mathbf{u}_e^T\mathbf{H}_e\mathbf{u}_e\right)$$ subject to $t \in B,$ where $p$ is a penalization parameter (typically, $p=3$).

\section{Computational simulations and results}

We present numerical results in an analogous two-dimensional context, more specifically for  two-dimensional beams of dimensions $1 \times l $ (units refer to the international system) represented by $\Omega=[0,1]\times [0,l]$, with $l=0.5$, $F=-10^6$ for the first case,  $l=0.5$ and $F=-10^7$ for the second one,
 $l=0.6$ and $F=-10^6$ for the third case and,  $l=1$ and $F=-10^8$ for the fourth one. $F$ is in the $y$-direction and corresponds to $f$ of the theoretical formulation presented above.

We consider the strain tensor as
$$e(\mathbf{u})=(e_x(\mathbf{u}),e_y(\mathbf{u}),e_{xy}(\mathbf{u}))^T,$$ where $\mathbf{u}=(u,v) \in W^{1,2}(\Omega;\mathbb{R}^2),$
 $e_x(\mathbf{u})= u_x$, $e_y(\mathbf{u})=v_y$ and $e_{xy}(\mathbf{u})=\frac{1}{2}(u_y+v_x).$

Moreover the stress tensor $\sigma(e(\mathbf{u}))$ is given by
$$\sigma(e(\mathbf{u}))=H(t) e(\mathbf{u}),$$ where
\begin{equation}H(t)= \frac{E(t)}{1-\nu^2}\left\{
\begin{array}{lcr}
 1 & \nu & 0\\
 \nu& 1& 0 \\
 0 & 0& \frac{1}{2}(1-\nu)\end{array} \right\}\end{equation}
 and $$E(t)=t E_0+(1-t) E_1,$$
 where $E_0=210 * 10^9$ (the modulus of Young) and $E_1\ll E_0.$ Moreover $\nu=0.33$.

 As previously mentioned, we present four numerical simulations.

\vspace{0.2cm}
\noindent{\bf Case 1}. For the first case see figure \ref{figure1}, on the left, for the concerning case, figure \ref{figure1}, in the middle,  for the optimal topology for  this  case with no filter, figure \ref{figure1}, on the right, for the optimal topology
 for this first case with filter. For the objective function as function of number of iterations also for such a case with no filter, see figure \ref{figure2}, on the left,  and for the objective function as function of number of iterations also for this first case with filter, see figure \ref{figure2}, on the right.
\begin{figure}[!htb]
\begin{center}
	\includegraphics[height=2.8cm]{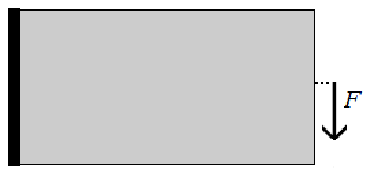}
    \includegraphics[height=3.2cm]{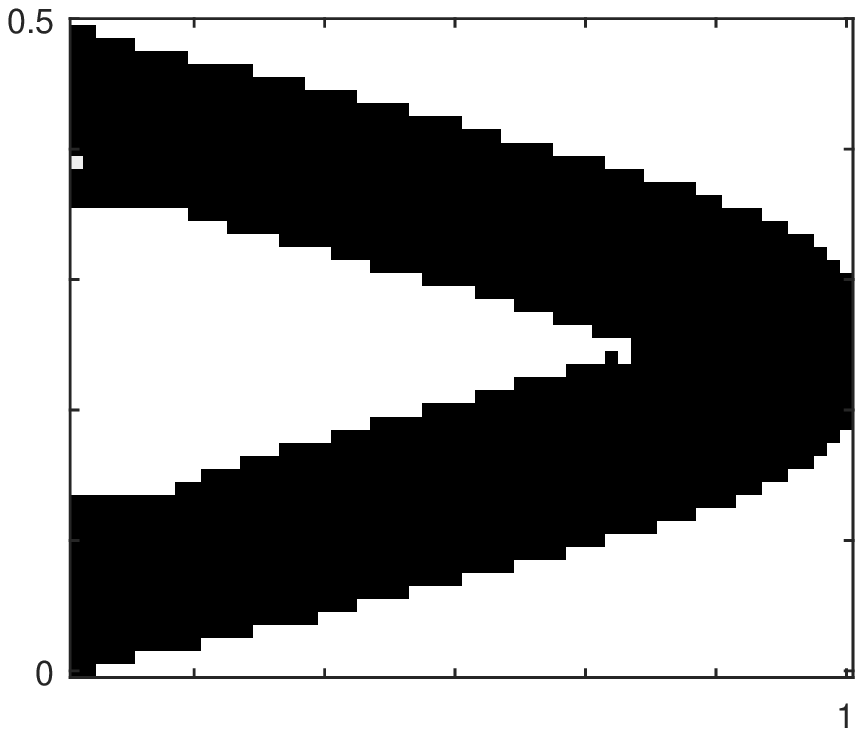}
    \includegraphics[height=3.2cm]{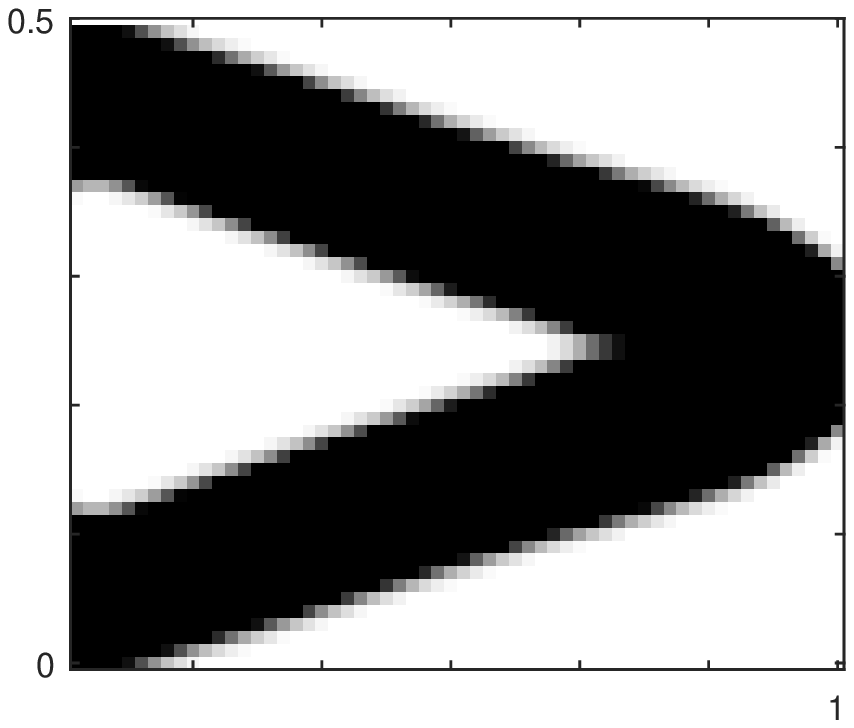}
    \vspace{-0.3cm}
\caption{\small{on the left a clamped beam at $x=0$ (cantilever beam). In the middle the optimal topology for $t_1=0.5$, for the case with no filter. On the right the optimal topology  for $t_1=0.5$, for the case with filter. The FE mesh was 60x50.}} \label{figure1}
\end{center}
\end{figure}
\begin{figure}[!htb]
\begin{center}
	\includegraphics[height=6cm,width=7cm]{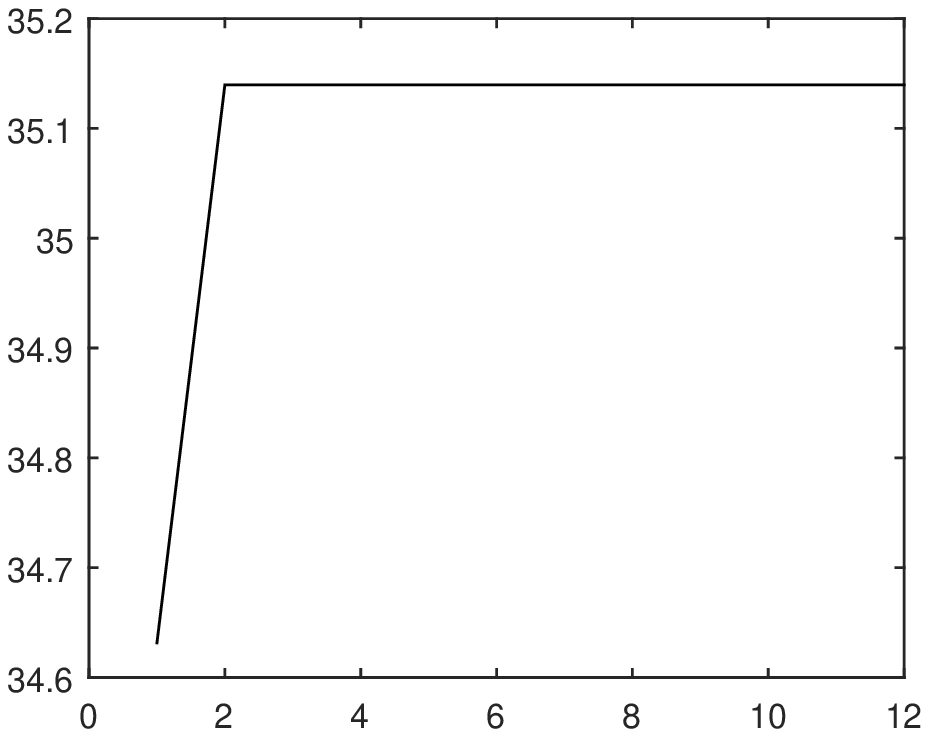}
    \includegraphics[height=6cm,width=7cm]{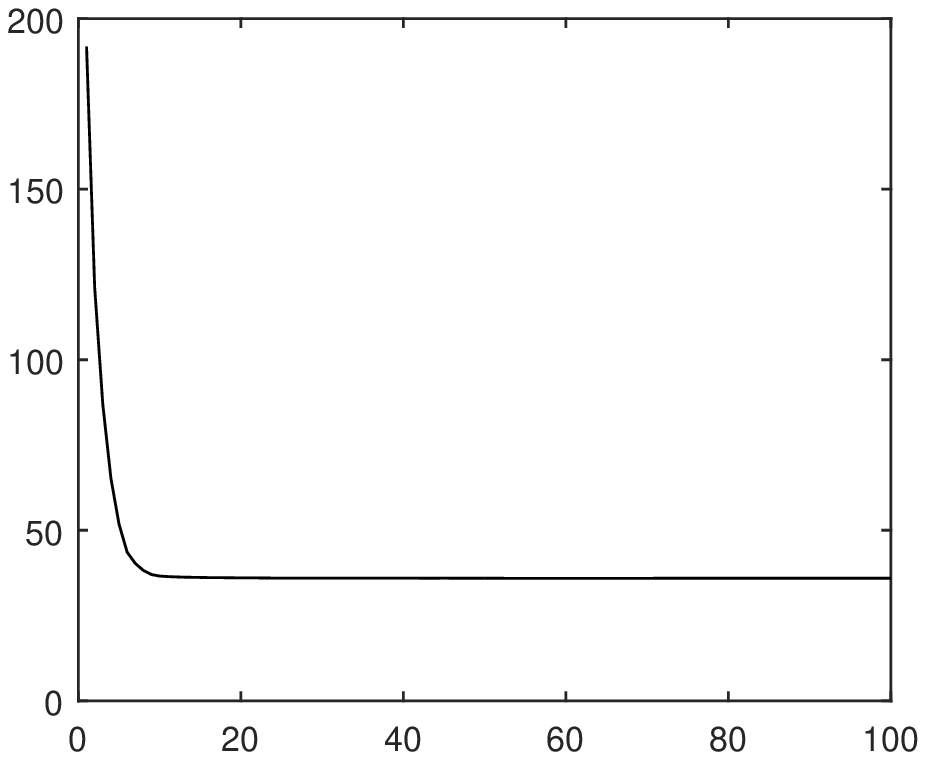}
    \vspace{-0.3cm}
 \caption{\small{on the left the objective function by iteration numbers for $t_1=0.5$, for the case with no filter. On the right the objective function by iteration numbers for $t_1=0.5$, for the case with filter.}} \label{figure2}
\end{center}
\end{figure}

\vspace{0.2cm}
\noindent{\bf Case 2}. For the second case see figure \ref{figure3}, on the left, for the concerning case, figure \ref{figure3}, in the middle,  for the optimal topology for  this  case with no filter, figure \ref{figure3}, on the right, for the optimal topology
 for this second case with filter. For the objective function as function of number of iterations also for such a case with no filter, see figure \ref{figure4}, on the left,  and for the objective function as function of number of iterations also for this second case with filter, see figure \ref{figure4}, on the right.
\begin{figure}[!htb]
\begin{center}
	\includegraphics[height=4cm]{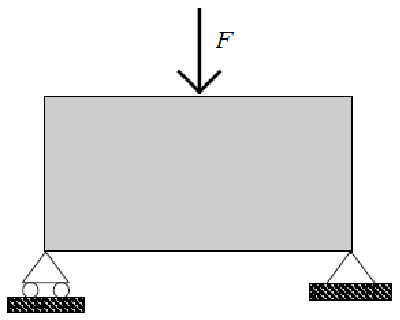}
    \includegraphics[height=3.5cm,width=5.5cm]{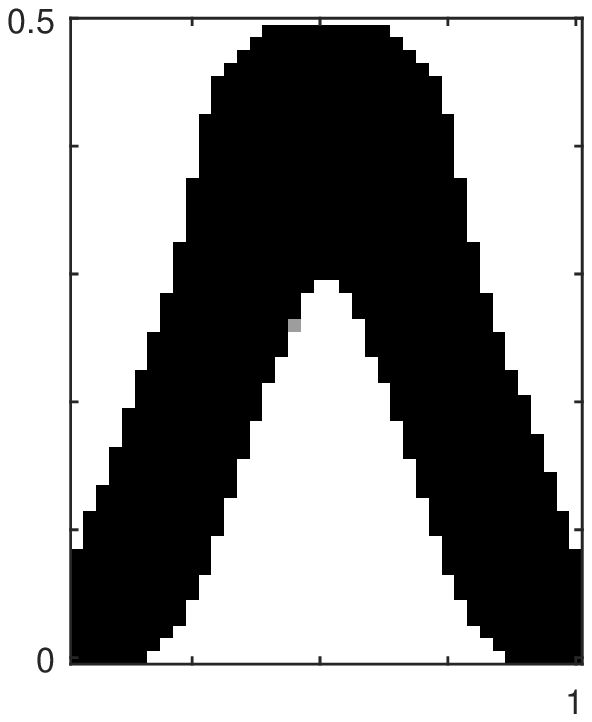}
    \includegraphics[height=3.5cm,width=5.5cm]{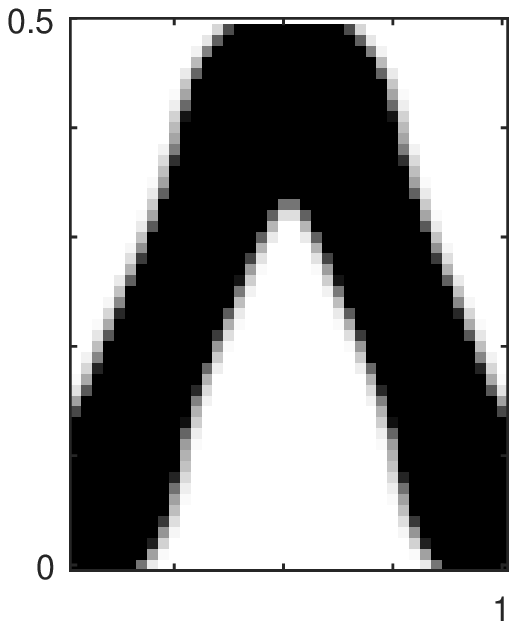}
    \vspace{-0.3cm}
\caption{\small{on the left a simply supported beam at $x=0$ and $x=1$. In the middle the optimal topology  for $t_1=0.5$, for the case with no filter. On the right the optimal topology  for $t_1=0.5$, for the case with filter. The FE mesh was 40x50.}} \label{figure3}
\end{center}
\end{figure}
\begin{figure}[!htb]
\begin{center}
	\includegraphics[height=5.5cm,width=7cm]{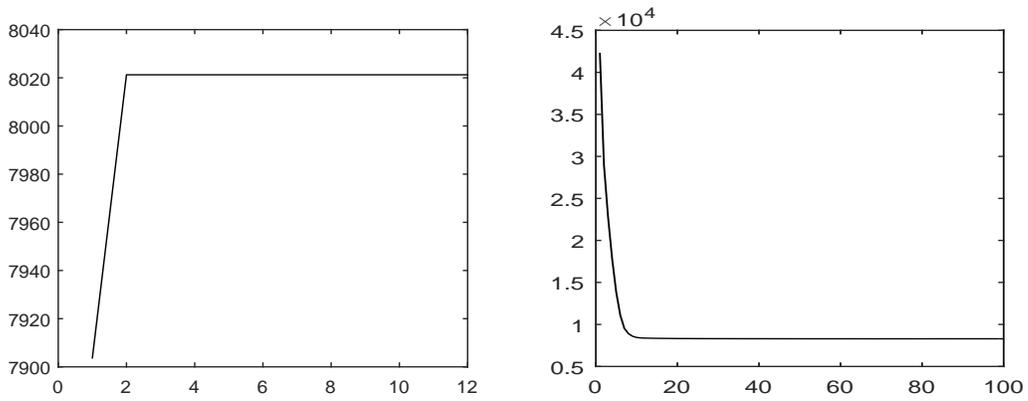}
    \includegraphics[height=5.5cm,width=7cm]{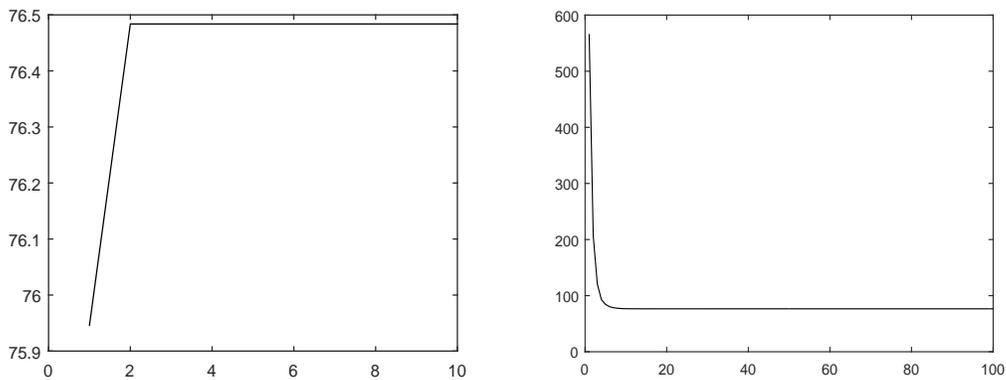}
    \vspace{-0.4cm}
\\ \caption{\small{on the left the objective function by iteration numbers for $t_1=0.5$, for the case with no filter. On the right the objective function by iteration numbers for $t_1=0.5$, for the case with filter.}} \label{figure4}
\end{center}
\end{figure}

\vspace{0.2cm}
\noindent{\bf Case 3}. For the third case see figure \ref{figure5}, on the left, for the concerning case, figure \ref{figure5}, in the middle,  for the optimal topology for  this case with no filter, figure \ref{figure5}, on the right, for the optimal topology
 for this third case with filter. For the objective function as function of number of iterations also for such a case with no filter, see figure \ref{figure6}, on the left,  and for the objective function as function of number of iterations  also for this third case with filter, see figure \ref{figure6}, on the right.
\begin{figure}[!htb]
\begin{center}
	\includegraphics[height=3.5cm,width=5cm]{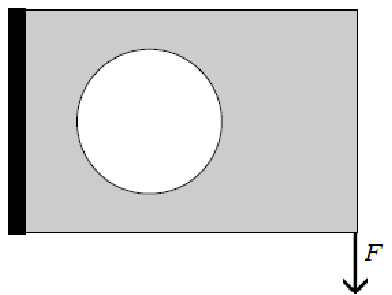}
    \includegraphics[height=3.5cm,width=5cm]{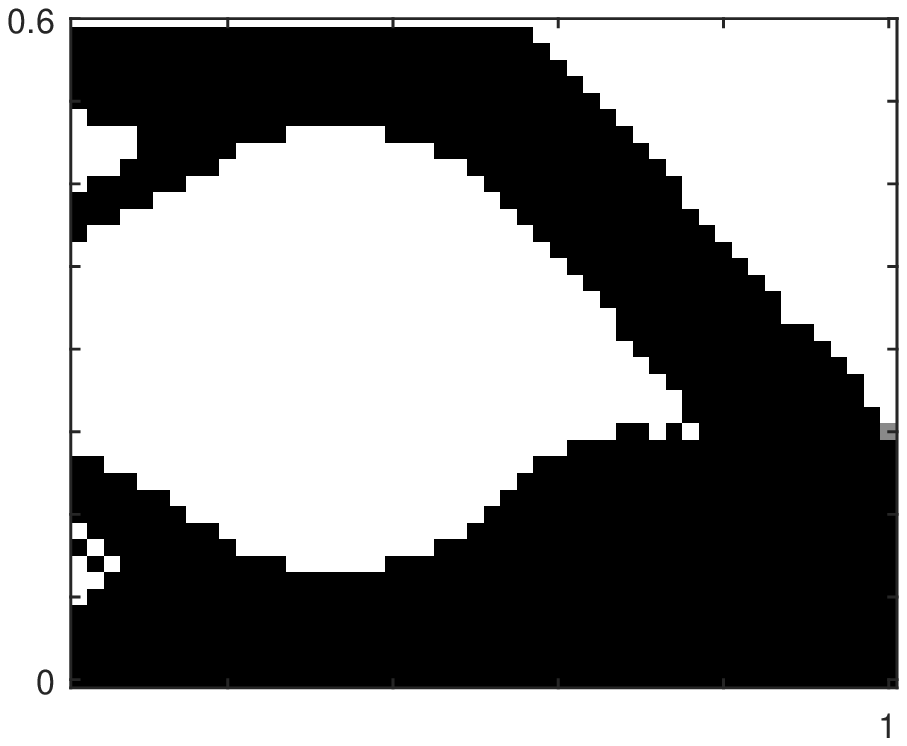}
    \includegraphics[height=3.5cm,width=5cm]{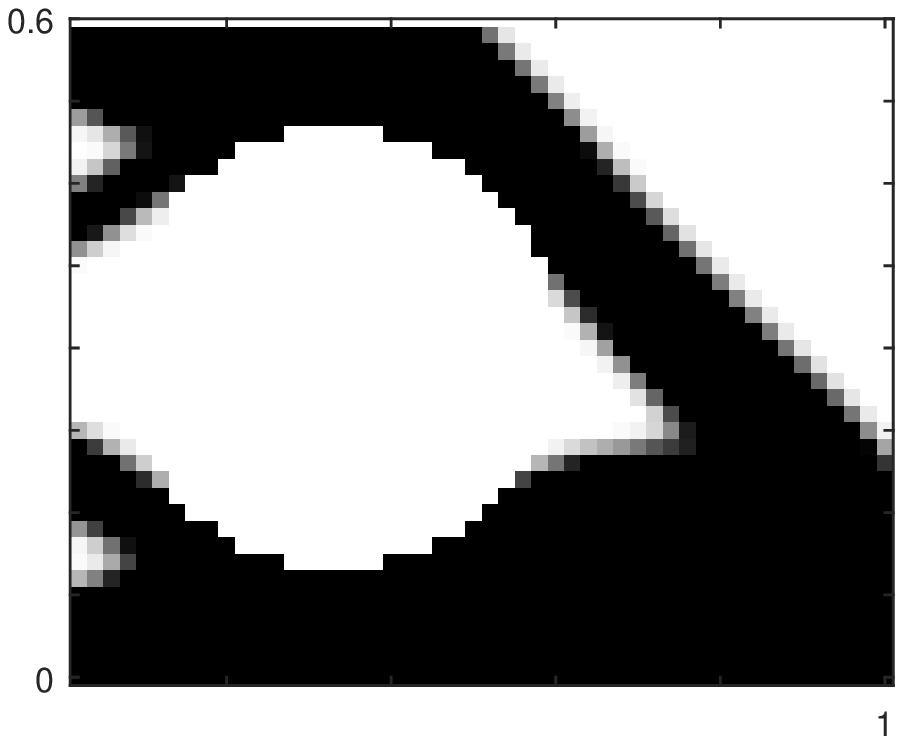}
    \vspace{-0.4cm}
\caption{\small{on the left a beam with a hole clamped at $x=0$. In the middle the optimal topology  for $t_1=0.5$, for the case with no filter. On the right the optimal topology  for $t_1=0.5$, for the case with filter. The FE mesh was 50x40.}} \label{figure5}
\end{center}
\end{figure}
\vspace{-0.5cm}
\begin{figure}[!htb]
\begin{center}
	\includegraphics[height=5.5cm,width=7cm]{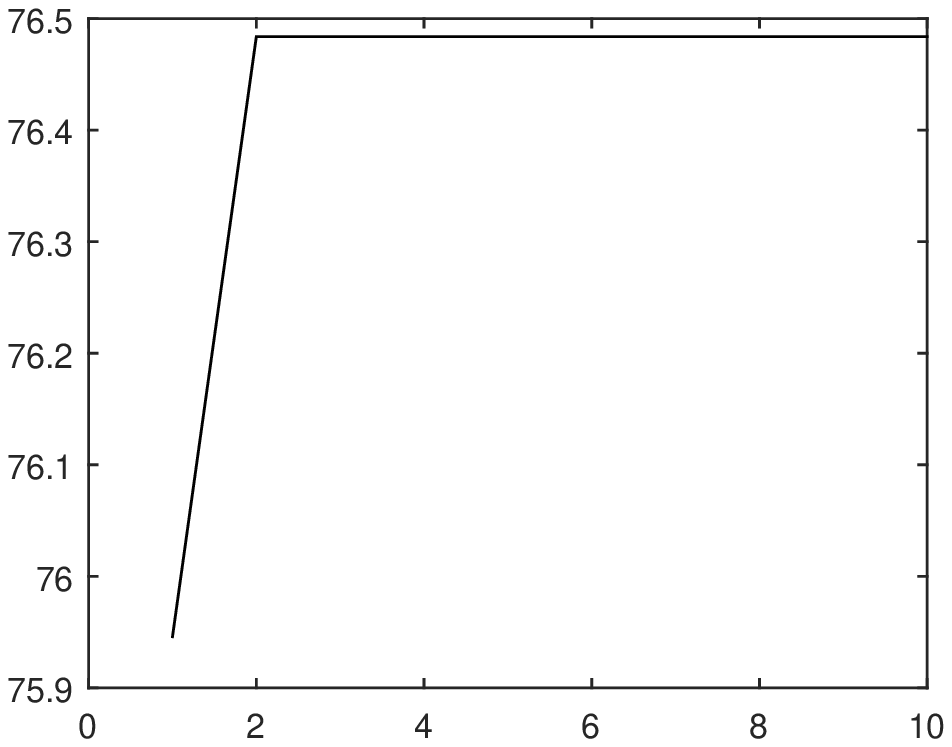}
    \includegraphics[height=5.5cm,width=7cm]{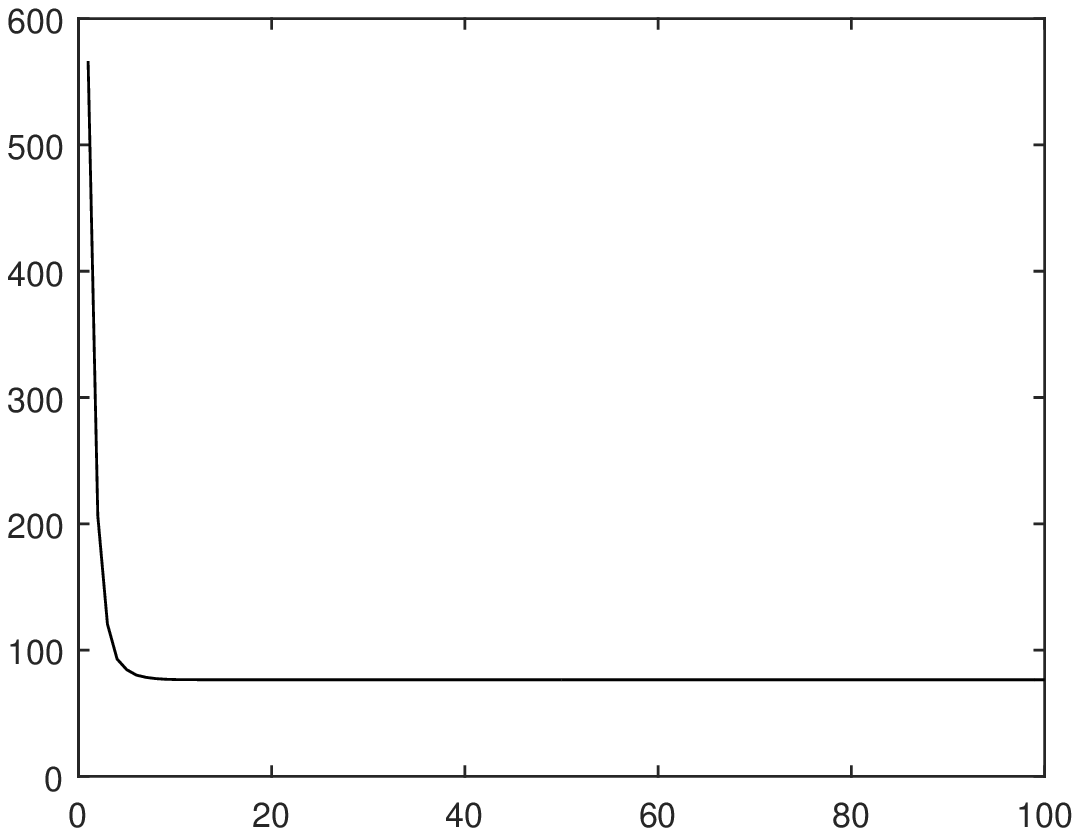}
    \vspace{-0.4cm}
\\ \caption{\small{on the left the objective function by iteration numbers for $t_1=0.5$, for the case with no filter. On the right the objective function by iteration numbers for $t_1=0.5$, for the case with filter.}} \label{figure6}
\end{center}
\end{figure}

\vspace{0.2cm}
\noindent{\bf Case 4}. For the fourth case see figure \ref{figure7}, on the left, for the concerning case, figure \ref{figure7}, in the middle,  for the optimal topology for  this case with no filter, figure \ref{figure7}, on the right, for the optimal topology
 for this fourth case with filter. For the objective function as function of number of iterations also for such a case with no filter, see figure \ref{figure8}, on the left,  and the objective function as function of number of iterations also for this fourth case with filter, see figure \ref{figure8}, on the right.
\begin{figure}[!htb]
\begin{center}
	\includegraphics[height=3.5cm,width=4cm]{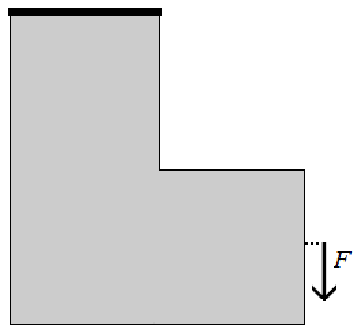}
    \includegraphics[height=4cm,width=6.5cm]{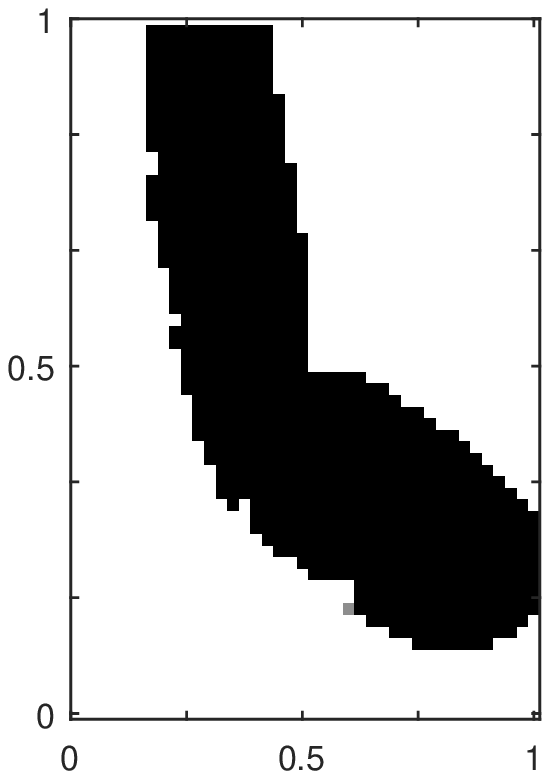}
    \includegraphics[height=4cm,width=5cm]{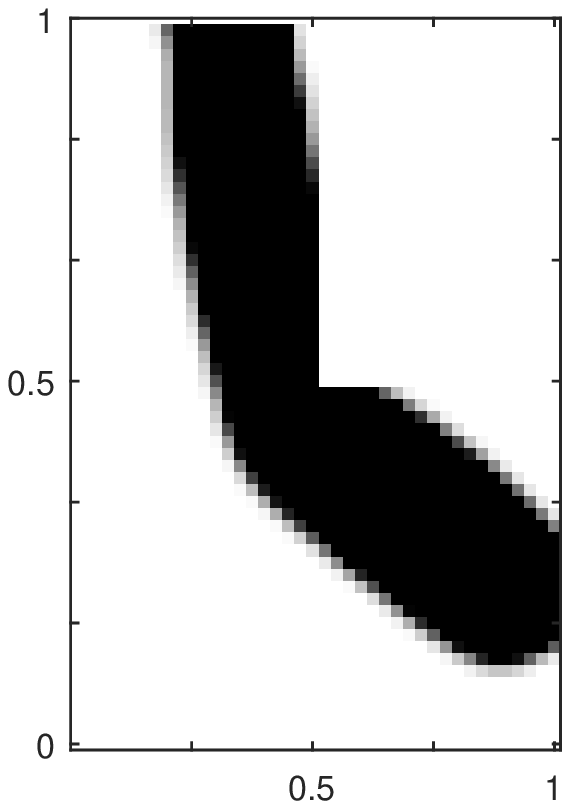}
    \vspace{-0.3cm}
\caption{\small{on the left a L shape beam clamped at $y=1$. In the middle the optimal topology  for $t_1=0.5$, for the case with no filter. On the right the optimal topology  for $t_1=0.5$, for the case with filter. The FE mesh was 40x60.}} \label{figure7}
\end{center}
\end{figure}
\begin{figure}[!htb]
\begin{center}
	\includegraphics[height=6cm,width=7cm]{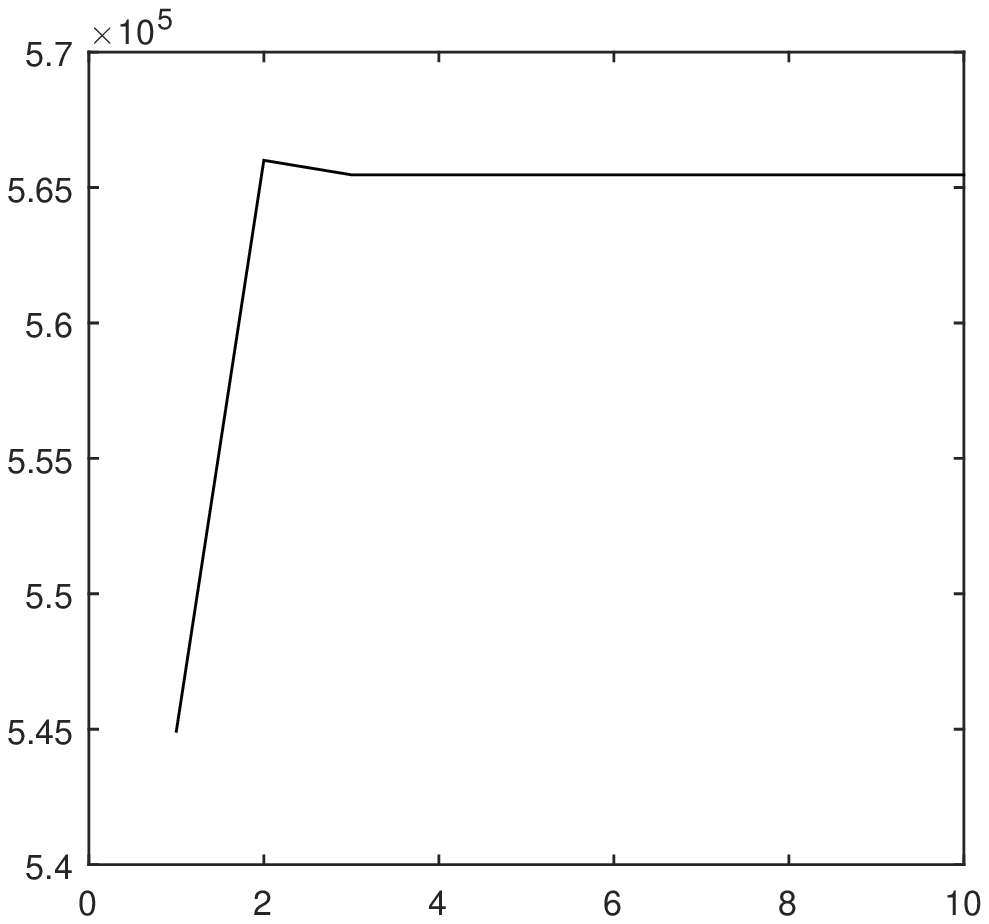}
    \includegraphics[height=6cm,width=7cm]{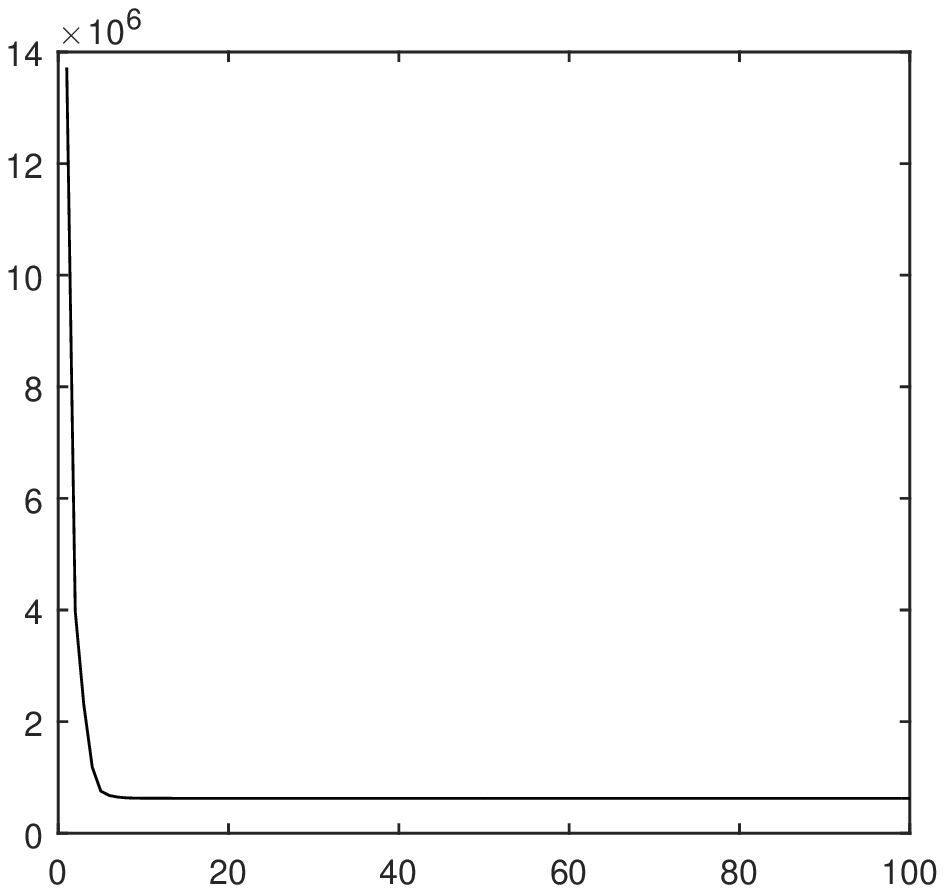}
    \vspace{-0.3cm}
\\ \caption{\small{on the left the objective function by iteration numbers for $t_1=0.5$, for the case with no filter. On the right the objective function by iteration numbers for $t_1=0.5$, for the case with filter.}} \label{figure8}
\end{center}
\end{figure}

We emphasize to have obtained in both optimized structures, without filter and with filter, robust topology from a structural point of view. One can note also in the figures that in all cases the objective functions, without filter and with filter, have similar final value, which indicates that the results obtained are consistent.

\section{Final remarks and conclusions}

In this article we have developed a duality principle and relating computational method for a class of structural optimization problems
in elasticity. It is worth mentioning we have not used a filter to post-process the results, having obtained a solution $t:\Omega \rightarrow \{0,1\}$
(that is, $t(x,y)=0$ or $t(x,y)=1$ in $\Omega$), by finding a critical point $(u_0,t_0) \in U \times B$ for the functional $J_1:U \times B \rightarrow \mathbb{R}.$ This corresponds, in some sense, to solving the dual problem.

We address some final remarks and conclusions on the results obtained.
\begin{itemize}
\item For all examples, in a first step, we have obtained numerical results through our algorithm with a software which uses the Matlab-Linprog as optimizer  at each iteration without any filter. In a second step, we obtain numerical results using the OC optimizer with filter, with a software developed based
in the  article \cite{3.1} by Sigmund, 2001.
\item We emphasize, to obtain good and consistent results, it is necessary to discretize more in the direction $y$, that is, the load direction, in which the displacements are much larger.
\item If we do not discretize enough in the load direction, for the software with no filter, a check-board standard in the material distribution is obtained in some  parts of the concerning struture.
\item Summarizing, with no filter, the check-board problem is solved by increasing the discretization in the load direction.
\item Moreover, with the OC optimizer with filter, the volume fraction of  material is kept constant in 0.5 at each iteration during the optimization process, whereas for the case with no filter we start with a volume fraction of 0.95 which is gradually decreased to the value 0.5, using  as the initial solution for  a iteration with a specific volume fraction, the solution of the previous one.
\item We also highlight the result obtained with no filter is indeed a critical point for the original optimization problem, whereas
there is some heuristic in the procedure with filter.
\item Once more we emphasize to have obtained more robust and consistent shapes by properly discretizing the approximate model in a
FE context.
\item Finally, it is also worth mentioning, we have obtained similar final objective function values without and with filter in all examples, even though without filter such values have been something smaller, as expected. The qualitative differences between the graphs without and with filter, for the objective function as function of the number of iterations, refer to the differences between the optimization
processes, where in the case with filter the volume fraction is kept  0.5 and without filter it is gradually decreased from 0.95 to 0.5, as above described.
\end{itemize}

We highlight the results obtained may be applied to other problems, such other models of plates, shells and elasticity.

\end{document}